\documentclass[10pt,a4paper]{article}
\usepackage[nice]{nicefrac}
\usepackage{cite}
\usepackage{amscd}
\usepackage[latin1]{inputenc}
\usepackage{amsmath}
\usepackage{amsfonts}
\usepackage{amssymb}
\usepackage{color}
\usepackage{float}
\usepackage{amsthm}
\usepackage{graphicx}
\usepackage{listings}
\usepackage{hyperref}
\usepackage{mathtools}

\newtheorem*{theorem*}{Theorem}
\newtheorem{theorem}{Theorem}

\newtheorem{lemma}[theorem]{Lemma}
\newtheorem{proposition}[theorem]{Proposition}
\newtheorem{remark}[theorem]{Remark}

\DeclareMathOperator{\rk}{rk}

\def\Qed{\hfill\raisebox{.6ex}{\framebox[2.5mm]{}}\\[.15in]}

\def\m{\mathbb}

\makeatletter
\newcommand{\xdashrightarrow}[2][]{\ext@arrow 0359\rightarrowfill@@{#1}{#2}}
\newcommand{\xdashleftarrow}[2][]{\ext@arrow 3095\leftarrowfill@@{#1}{#2}}
\newcommand{\xdashleftrightarrow}[2][]{\ext@arrow 3359\leftrightarrowfill@@{#1}{#2}}
\def\rightarrowfill@@{\arrowfill@@\relax\relbar\rightarrow}
\def\leftarrowfill@@{\arrowfill@@\leftarrow\relbar\relax}
\def\leftrightarrowfill@@{\arrowfill@@\leftarrow\relbar\rightarrow}
\def\arrowfill@@#1#2#3#4{%
  $\m@th\thickmuskip0mu\medmuskip\thickmuskip\thinmuskip\thickmuskip
   \relax#4#1
   \xleaders\hbox{$#4#2$}\hfill
   #3$%
}
\makeatother

\renewcommand\arraystretch{0.5}
\renewenvironment{proof}{\noindent{\bf Proof.}\newline}{\Qed}

\begin{document}
\date{}
\title{Explicit Schoen surfaces}
\author{Carlos Rito \and Xavier Roulleau \and Alessandra Sarti}
\maketitle

\begin{abstract}
We give an explicit construction for the $4$-dimensional family of Schoen surfaces by computing equations
for their canonical images, which are $40$-nodal complete intersections of a quadric and the Igusa quartic in $\mathbb P^4$.
We then study a particularly interesting example, with $240$ automorphisms and maximal Picard number.
%We geometrically construct Schoen surfaces as double covers of degree $8$ forty-nodal complete intersection surfaces for which we give the explicit equations.
%We then study a particularly interesting example.


\end{abstract}


\section{Introduction}

While working on a problem related to the Hodge conjecture, Chad Schoen \cite{Sc}
used deformation theory to construct a family of surfaces $S$
(from now on called {\it Schoen surfaces}) with some interesting
features. The one that has interested Schoen most is that the Albanese map
of $S$ embeds it into its Albanese variety $A$, but for a generic Schoen surface the cycle $S$
in $H^{4}(A,\mathbb Q)$ is not contained in the subspace generated
by the intersections of two divisors of $A;$
the existence of such cycles on an abelian variety $A$ makes the Hodge conjecture more difficult to prove.

Another interesting property of $S$ is that the natural map 
\[
\wedge^{2}H^{0}(S,\Omega_{S})\to H^{0}(S,K_{S})
\]
has a one dimensional kernel, and therefore $S$ is a Lagrangian surface
in $A$. Moreover the kernel is not of the form $\omega_{1}\wedge\omega_{2}$,
($\omega_{i}\in H^{0}(S,\Omega_{S})$), therefore by the Castelnuovo-De
Franchis Theorem, the surface do not admit fibrations onto a curve
of genus $\geq2$. Only a few number of Lagrangian surfaces without
a fibration onto a curve of higher genus are known (see \cite{BNP, BPS, BT}).
Such examples are interesting for people studying k\"ahelerian groups,
e.g. one can ask whether their fundamental group is nilpotent (cf.
\cite{Ca}).

By \cite{Be1}, when the canonical map of a surface of general
type has degree $>1$ onto a surface, that surface either has $p_{g}=0$
or is itself canonically embedded, the latter case being rather exceptional
(see \cite{CPT} for a list of the examples known so far). In \cite{CMR},
Ciliberto, Mendes Lopes and the second author studied Schoen surfaces geometrically,
proving that the canonical map of a Schoen surface $S$ is $2$-to-$1$ onto
a $40$-nodal degree $8$ complete intersection surface
$X_{40}\subset\mathbb P^{4}$ and the ramification of the double cover $S\to X_{40}$ is  the
set of $40$ nodes.
They also show that Schoen surfaces are not universally covered by the bidisk
(very few surfaces with $K^2=8\chi$ and such property are known).

Miyaoka's bound tells us that on a degree $8$ complete intersection
surface in $\mathbb P^{4}$, there cannot be more than $40$ nodes.
The construction of Schoen surfaces gives the first theoretical
proof  that such $40$-nodal surface exists, but without providing any
equations for it. 
In \cite{Be3}, Beauville used Schoen surfaces in order to
show the existence of $48$-nodal degree $16$ complete intersection
surfaces $X_{48}$ in $\mathbb P^{6}$ and surfaces $\tilde S$ whose canonical map
is $2$-to-$1$ onto $X_{48}$. 

%Schoen's construction is surprising: indeed in order to obtain surfaces with the same invariants and interesting properties as Schoen
%surfaces, some experts were first trying to construct a $40$-nodal degree $8$ complete intersection $X\subset\mathbb P^{4}$.

The main result of this paper is an explicit construction of the surfaces
$X_{40}$ by equations, and therefore an explicit construction of $S$
by double cover. The idea of the construction of the surfaces $X_{40}$ is the following.
The Igusa quartic threefold $I_{4}\subset\mathbb P^{4}$ is singular along $15$
lines of $A_{1}$ singularities. Its intersection with a generic quadric
gives thus a degree $8$ complete intersection surface containing
$30$ nodes. The main question is therefore to find a quadric $Q_{2}$ which,
while still transversal to the $15$ singular lines, is tangent to the Igusa quartic at $10$ more points,
leading to a $40$-nodal surface $X_{40}:=I_{4}\cap Q_{2}$. Our construction
is very concrete, since we have the explicit equation
for the quadric $Q_{2}$, whose coefficients are depending of $4$
parameters.

In order to obtain that result, we use the knowledge of the above
mentioned papers, computer algebra and the rich geometry of the Igusa
quartic threefold, and of its dual, the Segre cubic threefold $S_{3}\subset\mathbb P^{4}$,
which is the unique cubic threefold with $10$ nodes.
It is well known that the Igusa quartic threefold parametrizes quartic Kummer
surfaces \cite[Theorem 3.3.8]{Hu}: taking the intersection of
$I_{4}$ by a hyperplane $T_{x}$ tangent to a generic point $x\in I_{4}$,
one obtains a $16$-nodal Kummer surface, $15$ nodes coming from
the intersection of the $15$ lines in $I_{4}$ with $T_{x}$, and
one more node at $x$. Our construction of $X$ uses the $15$-nodal
$K3$ surfaces obtained as the intersections of $I_{4}$ with a generic
hyperplane, giving new interesting geometric features to the Igusa
quartic $I_{4}$.

We then study a Schoen surface with a large group of symmetries (of
order $240$). We compute the isogeny class of its Albanese variety and we prove that it has maximal Picard number. Although interesting, examples of surfaces with maximal Picard number are rather scarce, see e.g. \cite{Be2}.

As a by product of our work, we also obtain a geometric construction of
a $3$-dimensional subfamily of Schoen surfaces, as a bidouble cover
of some particular Kummer surfaces. Construction which interestingly
matches some (theoretical) constructions of Lagrangian surfaces suggested
by Bogomolov and Tschinkel in \cite{BT} (see Remark \ref{RemSubfamily}).

Since $15$-nodal quartic surfaces, obtained as generic hyperplane sections of the Igusa quartic,
play a key role in our construction, one may ask if an analogous construction could be done
using a different family of $15$-nodal quartics. The answer is negative: we show in Appendix \ref{Sarti} that a generic
quartic surface with $15$ nodes can be realized as a hyperplane section of the Igusa quartic threefold.

The paper is organized as follows.
In Section 2, we recall some known facts on Schoen surfaces. % which will be used later. 
 In Section 3, we construct the $40$-nodal degree
$8$ surfaces in $\mathbb P^{4}$, we prove that their set of $40$ nodes
is $2$-divisible, their associated double covers are not universally
covered by the bidisk and that they are Schoen surfaces.
In Section 4, we study an example of a Schoen surface with a large
automorphism group. The Appendix is on the moduli of K3 surfaces
with $15$ nodes.% the second contains computations using the
%computer algebra system Magma \cite{BCP}.

\bigskip
\noindent{\bf Notation}

We work over the complex numbers. All varieties are assumed to be projective algebraic.
For a smooth surface $S,$ as usual $K_S$ is the canonical class, $p_g(S):=h^0(S,K_S)$ is the geometric genus,
$q(S):=h^1(S,K_S)$ is the irregularity and $\chi(\mathcal O_S)=1-q+p_g$ is the holomorphic Euler characteristic.
A $(-n)$-curve on a surface is a curve isomorphic to $\m P^1$ with self-intersection $-n.$
Linear equivalence of divisors is denoted by $\equiv.$
%The rest of the notation is standard in Algebraic Geometry.\\

\bigskip
\noindent{\bf Acknowledgements}

The first author thanks the university of Poitiers for the hospitality during his visit in March 2016.
This research was partially supported by FCT (Portugal) under the project PTDC/MAT-GEO/2823/2014,
the fellowship SFRH/ BPD/111131/2015 and by CMUP (UID/MAT/00144/2013),
which is funded by FCT with national (MEC) and European structural funds through the programs FEDER, under the partnership agreement PT2020.

The authors would like to thank Amir Dzambic, Alice Garbagnati, Margarida Mendes Lopes, Rita Pardini, Pierre Py and Chad Schoen for useful conversations or correspondence.

\section{Schoen surfaces}

Let $C$ be a smooth genus $2$ curve with
jacobian $J(C)$ and consider the union $$V:=C\times C\, \cup_{C}\, J(C)$$
glued along the diagonal of $C\times C$ and $C\hookrightarrow J(C).$
Notice that $V$ is singular along $C.$
\begin{theorem}[\cite{Sc}]\label{thmSc}
The reducible surface $V$ can be deformed into a smooth surface of general type $S$ with invariants
\[
c_{1}^{2}=16=2c_{2},\ q=4,\ p_{g}=5.
\]
The moduli of these surfaces is $4$-dimensional.
The deformation space is locally smooth, thus it is locally irreducible.
\end{theorem}
\noindent As said in the Introduction, the canonical map of a Schoen surface $S$ is of degree 2 onto a $40$-nodal complete intersection
$X$ of a quadric and a quartic in $\mathbb P^4.$
From \cite[Lemmas 4.5, 4.6 and proof of Theorem 4.1]{CMR} we deduce the following.
\begin{proposition}[\cite{CMR}]\label{propCMR}
The above surfaces $X$ degenerate to the union of a double quadric surface and a quartic Kummer surface,
glued along a trope of the Kummer surface.
These surfaces are given by the intersection of two hyperplanes and a quartic hypersurface in $\mathbb P^4.$
Moreover, this degeneration induces the degeneration in Schoen's construction.
\end{proposition}

\section{The construction}\label{construction}
In this section we show the following:
\begin{theorem}
Let $I_4$ be the Igusa quartic in $\mathbb P^4.$
There exists a quadric on $4$ parameters $Q_{a,b,c,d}$ such that, for generic  values of the parameters,
the surface $$X_{40}:=I_4\cap Q_{a,b,c,d}$$ has exactly $40$ nodes.
These nodes are $2$-divisible in the Picard group and the double cover $S\rightarrow X_{40}$ ramified
over the nodes is a Schoen surface.
\end{theorem}
We explain how to compute the quadric $Q_{a,b,c,d}.$
The corresponding computer code, implemented with Magma, is available at \cite{Ri}.% \url{http://www.crito.utad.pt/schoen.pdf}. 

\subsection{Segre cubic, Igusa quartic}\label{SI}

The linear system $L$ of quadrics through points $p_1,\ldots,p_5\in\mathbb P^3$ in general position (i.e. no $4$ of them are contained in a hyperplane)
is $4$-dimensional. Let \mbox{$\phi:\mathbb P^3\dashrightarrow\mathbb P^4$} be the rational map corresponding to the linear system $L$ on $\mathbb{P}^3$.
The image $S_3:=\phi(\mathbb P^3)$ is the {\em Segre cubic}, the unique cubic threefold in $\mathbb P^4$ (up to projective equivalence)
with singular set the union of $10$ nodes (the images of the lines $p_ip_j$).
The Segre cubic contains $15$ planes:
the "images" (after blowing up $\mathbb P^3$) of $p_1,\ldots,p_5$ and of the
$10$ planes in $\mathbb P^3$ through exactly $3$ of the points $p_i.$
The dual variety $I_4$ (the image under the gradient map) of $S_3$ is the {\em Igusa quartic}.
The dual map contracts the above $15$ planes to singular lines of $I_4$, its singular set.
The Igusa quartic has $10$ {\em tropes}, i.e. $10$ hyperplane sections which are double quadrics.
For more details see e.g. \cite{Hu} or \cite{Do2}.

\subsection{The $40$-nodal surface}\label{40nodal}

Let $H\subset\mathbb P^4$ be a generic hyperplane. Then $Q_{15}:=I_4\cap H\subset\mathbb P^3$ is a quartic surface with $15$ nodes.
Since $H$ is generic, we can choose five nodes $p_1,\ldots,p_5$ in general position (i.e. no $4$ of them are in a hyperplane).
Consider the map $\phi:H\dashrightarrow\mathbb P^4$ given by the linear system $|L|$ of quadrics which pass through $p_1,\ldots,p_5.$
\begin{proposition}\label{completeint}
There exists a quadric $S_2\subset\mathbb P^4$ such that $$Q_{10}:=\phi(Q_{15})\cong S_3\cap S_2$$
and $Q_{10}$ has $10$ nodes which are disjoint from the nodes of $S_3.$
\end{proposition}

\begin{proof}
We have $\phi(H)\cong S_3$ and $\phi$ is of degree $1$ outside of the lines $p_ip_j,$ therefore
$\phi$ sends $Q_{15}$ birationally to a surface $Q_{10}$ contained in $S_3,$ a $10$-nodal $K3$ surface.
Now consider the resolution of singularities $\widetilde Q_{15}\rightarrow Q_{15}$
and let $|L'|$ be the strict transform of $|L|$ in $\widetilde Q_{15}$.
Since $L'^2=6$ and $|L'|$ has no base points, \cite[Theorem 6.1]{Sd} implies that $Q_{10}$ is a complete intersection of a
quadric $S_2$ and a cubic in $\mathbb P^4.$ This cubic can be assumed to be $S_3$ because $Q_{10}\subset S_3.$

The second assertion follows from the fact that the $10$ nodes of $Q_{10}$ correspond to the nodes of $Q_{15}$
disjoint from the $10$ lines $p_ip_j,$ $i,j\in\{1,\ldots,5\},$ which are contracted to the nodes of $S_3.$
\end{proof}

Using Magma, we compute  %\ref{magmacodeSec1}
this $4$-dimensional family of smooth quadrics $S_2$ (see \cite[Section A]{Ri}).
The dual of $S_2$ is a smooth quadric $Q_2.$ Consider the dual maps
$$d_1:S_3\dashrightarrow I_4,$$ $$d_2:S_2\dashrightarrow Q_2$$ and define $X_{40}:=I_4\cap Q_2.$
Since $S_2$ is tangent to $S_3$ at $10$ smooth points of $S_3$ and duality preserves tangencies, then
$X_{40}$ has at least $10$ singular points. The purpose of this construction is to find $Q_2$ meeting the
$15$ singular lines of $I_4$ transversally, so that $X_{40}$ is a $40$-nodal surface.
We show below that, up to the symmetry of $I_4,$ at most one choice of the nodes $p_1,\ldots,p_5$ serves our aims.
Notice that the quartic surface $Q_{15}$ has $10$ tropes (double conics), which are induced by the tropes of $I_4.$
\begin{proposition}
If exactly three of the nodes $p_1,\ldots,p_5$ are in a trope of $Q_{15},$ then the surface $X_{40}$ is non-normal.
\end{proposition}

\begin{proof}
Let $T$ be the hyper plane of $H$ which gives the trope of $Q_{15}$ containing three of the nodes $p_1,\ldots,p_5.$
Since $\phi$ is injective outside the union of the lines $p_ip_j$, one has $\phi(T\cap Q_{15})=\phi(T)\cap S_2.$
Let $C$ be the conic such that $T\cap Q_{15}=2C.$
After blowing up $\mathbb P^3$ at $p_1,\ldots,p_5,$ the strict transform of a quadric in $|L|$ meets the strict transform
of $C$ at exactly one point. This implies that the image $\phi(C)$ is a line. 
We have then that $\phi(T)\cap S_2$ is a double line and, as stated in Section \ref{SI}, $\phi(T)$ is a plane in $S_3.$
When taking the duals, this plane is contracted to a singular line of $I_4$ (see e.g. \cite[$\S$ 3.3.4]{Hu}) and the quadric $Q_2$ contains this line.
This implies that the singular set of $X_{40}$ is of dimension $1.$
\end{proof}

Therefore, a surface $X_{40}$ is normal only if no $3$ of the nodes $p_1,\ldots,p_5$ are in a trope of $Q_{15}.$
We compute that there are exactly $6$ such sets of nodes (see \cite{Ri}). The $10$ tropes of $I_4$
induce $10$ tropes of $Q_{15},$ so we compute the sets of $5$ singular lines of $I_4$ such that there is no trope
containing $3$ of them. Fixing one of these sets (the $S_6$ symmetry of $I_4$ gives the remaining sets) we have
a choice of five nodes $p_1,\ldots,p_5$ for each surface $Q_{15}.$ The computations with Magma available in \cite{Ri} confirm that a generic surface $X_{40}$ constructed as above has $40$ nodes and no other singularities.
Our computations are optimal in the sense that we construct the entire family at once:
the output is a quadric on $4$ parameters $Q_{a,b,c,d}$ such that, for generic  values of the parameters,
the surface $I_4\cap Q_{a,b,c,d}$ has exactly $40$ nodes.

\begin{figure}[h]
$$
    \begin{array}{ccccccc}
     & & S_3 & \xdashrightarrow{\text{\ \ \ \ \ $d_1$ \ \ \ \ }} & I_4 & \\
    I_4\cap H_{a,b,c,d} & \xdashrightarrow{\text{\ \ \ \ \ $|L|$ \ \ \ \ }} & \supset &  & \supset & =:X_{40}  \\
     & & S_2 & \xdashrightarrow{\text{\ \ \ \ \ $d_2$ \ \ \ \ }} & Q_{a,b,c,d} & 
     \end{array}
$$
 \caption{Here the symbol $\supset$ means intersection.}
 \label{fig1}
\end{figure}

%%$$
%    \begin{array}{ccccccccccc}
%     & & S_3 & & & \xdashrightarrow{\text{\ \ \ \ $d_1$ \ \ \ }} & I_4 & & & \\
%    I_4\cap H_{a,b,c,d} & \xdashrightarrow{\text{\ \ \ \ $|L|$ \ \ \ }} & & \cap & &  & & \cap & & =:X_{40} \\
%     & & & & S_2 & \xdashrightarrow{\text{\ \ \ \ $d_2$ \ \ \ }} & & & Q_{a,b,c,d} & 
%     \end{array}
%$$

\subsection{$2$-divisibility of the nodes}\label{2div}

In the previous section we proved that to a generic point $t$ in $(\mathbb{P}^4)^*$ 
(the dual space of $\mathbb{P}^4$), one can associate a surface $X_t$, the minimal 
resolution of a $40$-nodal surface $X_{40}(t)$ in $\mathbb{P}^4$. 
Let $o$ be a point in $(\mathbb{P}^4)^*$ such that the surface $X_{o}$ is defined. 
There exists a small polydisk  $B$ centered at $o$, 
a smooth  complex manifold $\mathcal{X}$ and
 a proper flat morphism $\mathcal{X}\to B$ such that the fiber at $t\in B$  is $X_t$. 
 There are moreover $40$ divisors $D_1,\dots,D_{40}$ on $\mathcal X$ such that 
 the intersection of these divisors with $X_t$ are the $40$ $(-2)$-curves of $X_t$
  over nodes in $X_{40}(t)$. One has:

\begin{proposition}\label{2div1}
Suppose that the sum of the $40$ $(-2)$-curves on $X_{o}$ is $2$-divisible.
Then, up to shrinking $B$, the sum of the $40$ $(-2)$-curves on the surface $X_{t}$ ($t\in B$)
is $2$-divisible.
\end{proposition}

\begin{proof}
Let be $\mathcal{L}=\mathcal{O}_{\mathcal{X}}(\sum D_{i})$. By hypothesis there exists a line bundle
 $\mathcal{M}_o$ on $X_o$ such that $\mathcal{L}_{|X_o}=\mathcal{M}_o ^{\otimes 2}$. The results follows from  \cite[Lemma 2.2]{CMR}.
 %
%
%By  \cite[Chapter 5, Lemma 3.1]{Dimca}, the integral cohomology groups of a smooth complete
%intersection are torsion free. 
%The surface $X_t$ is the minimal resolution of a degree $8$ complete intersection with nodal points.
%By the Theorem of Brieskorn on simultaneous resolutions of ADE singularities \cite{Brieskorn},
%the group $H^2(X^t,\mathbb{Z})$ is isomorphic to the second cohomology group of a smooth complete intersection,
%in particular it is torsion free.
%Let $A_{1},\dots,A_{40}$
%be the $40$ $(-2)$-curves on the surface $X^{0}$. By the hypothesis,
%there exists a class $L_{0}$ in $NS(X^{0})=H^{2}(X^{0},\mathbb Z)\cap H^{1,1}(X^{0})$
%such that 
%\[
%\sum_{i=1}^{40}A_{i}\sim2L_{0},
%\]
%where the symbol $\sim$ denotes numerical equivalence.
%From the Ehresmann trivialization theorem,
%there exists a diffeomorphism
%$\phi_{t}:X^{t}\to X^{0}$ such that $\phi_{t}^{*}(A_{i})$ is a $(-2)$-curve on $X^{t}$
%(see e.g. \cite[Theorem 9.3, Remark 9.4]{Voisin}).
%The natural map
%\[
%\phi_{t}^{*}:H^{2}(X^{0},\mathbb Z)\to H^{2}(X^{t},\mathbb Z)
%\]
%is an isomorphism of lattices, therefore we have 
%\[
%\sum_{i=1}^{40}\phi_{t}^{*}(A_{i})\sim2\phi_{t}^{*}(L_{0}),
%\]
%with $\phi_{t}^{*}(L_{0})\in H^{2}(X^{t},\mathbb Z).$ Since $\sum_{i=1}^{40}\phi_{t}^{*}(A_{i})\in H^{1,1}(X^{t})$,
%the class $\phi_{t}^{*}(L_{0})$ is in $H^{1,1}(X^{t})$, thus $\sum_{i=1}^{40}\phi_{t}^{*}(A_{i})$
%is $2$-divisible.
\end{proof}

According to our computations (see \cite[Section B]{Ri}), %Appendix \ref{S5X40}, 
one of the $40$-nodal surfaces constructed above is
projectively equivalent to the $\Sigma_5$-invariant surface $\overline X_{40}$ given in $\mathbb P^5$ by
$$x+y+z+w+t+h=0,$$
$$5\left(x^2+\cdots+t^2\right)-7\left(x+\cdots+t\right)^2=0,$$
$$4\left(x^4+\cdots+t^4+h^4\right)-\left(x^2+\cdots+t^2+h^2\right)^2=0.$$
\begin{proposition}\label{2div2}
The nodes of $\overline X_{40}$ are $2$-divisible.
\end{proposition}

\begin{proof}
One can verify that $\overline X_{40}$ has a $(40,12)$ configuration: $40$ tropes and $40$ nodes,
each trope contains $12$ nodes, through each node pass $12$ tropes.
Using Magma (see \cite[Section C]{Ri}), %Appendix \ref{2divA} 
we show the existence of tropes $T_1,\ldots,T_4$ such that:
\begin{description}
\item[$\cdot$] $T_i=2C_i,$ with $C_2,C_3,C_4$ smooth and $C_1$ the union of two conics;
\item[$\cdot$] the singular points of $C_1$ are not in $C_2\cup C_3\cup C_4$;
\item[$\cdot$] $C_1\cup C_2$  contains exactly $20$ nodes of $\overline X_{40}$ which are not in $C_1\cap C_2$;
\item[$\cdot$] $C_3\cup C_4$  contains exactly $20$ nodes of $\overline X_{40}$ which are not in $C_3\cap C_4$;
\item[$\cdot$] the above two sets of $20$ nodes are disjoint.
\end{description}

Let $\widehat X_{40}$ be the smooth minimal model of $\overline X_{40}.$
Denote by $\widetilde T_i$ the total transform of $T_i$ in $\widehat X_{40}$ and
by $\widehat C_i$ the strict transform of $C_i$ in $\widehat X_{40},$ $i=1,\ldots,4.$
There are $(-2)$-curves $A_1,\ldots,A_{22}\subset\widehat X_{40}$ such that
$$\widetilde T_1=2\widehat C_1+\sum_1^{10}n_iA_i+n_{21}A_{21}+n_{22}A_{22},$$
$$\widetilde T_2=2\widehat C_2+\sum_{11}^{20}n_iA_i+n_{21}'A_{21}+n_{22}'A_{22},$$
for some integers $n_1,\ldots,n_{22},n_{21}',n_{22}'.$
From $$0=\widetilde T_j A_i=2\widehat C_j A_i-2n_i=2-2n_i,$$ we get $n_i=1,$ $i=1,\ldots,22,$ and also $n_{21}'=n_{22}'=1.$
So, we have $$2\widetilde T\equiv\widetilde T_1+\widetilde T_2=2\widehat C_1+2\widehat C_2+\sum_1^{20}A_i+2A_{21}+2A_{22},$$
where $\widetilde T$ is the pullback of a general hyperplane section of $\overline X_{40}.$
This shows that $C_1\cup C_2$ contains $20$ nodes of $\overline X_{40}$ which are $2$-divisible.
Analogously, the remaining $20$ nodes of $\overline X_{40},$ contained in $C_3\cup C_4,$ are also $2$-divisible.
\end{proof}
\begin{proposition}
The $40$ nodes of a generic surface $X_{40}$ are $2$-divisible.
\end{proposition}

\begin{proof}
Immediate from Propositions \ref{2div1} and \ref{2div2}.
\end{proof}

\subsection{Surfaces with $K_S^2=2c_2=16,\,q=4$}

From the previous section, for a surface $X_{40}$ with exactly $40$ nodes as constructed above there is a
double covering $$\pi:S\longrightarrow X_{40}$$ ramified over the nodes.
\begin{proposition}
We have $$p_g(S)=5,\ q(S)=4, \ K_S^2=16.$$
\end{proposition}

\begin{proof}
%From the adjunction formula, the canonical system of $X_{40}$ is induced by the system of hyperplanes of $\mathbb P^4,$thus it is free from base points. 
The canonical line bundle $K_S$ is the  the pullback of  $K_{X_{40}}$ which is nef, thus $S$ is minimal and $K_S^2=2K_{X_{40}}^2=16$.\\
Let $\widehat X_{40}$ be the smooth minimal model of $X_{40},$ $A_1,\ldots,A_{40}$ be the $(-2)$-curves
which contract to the nodes  of $X_{40}$ and $S'\rightarrow\widehat X_{40}$ be the double covering with
branch locus $\sum_1^{40} A_i.$
The minimal model of $S'$ is isomorphic to $S.$

Let $L$ be the divisor such that $\sum_1^{40} A_i\equiv 2L.$ The double covering formulas (see e.g. \cite[V. 22]{BHPV}) give
$$\chi(S)=2\chi\left(\widehat X_{40}\right)+\frac{1}{2}L\left(K_{\widehat X_{40}}+L\right)=12-10=2.$$
%$$K_S^2=2\left( K_{\widehat X_{40}}+L \right)^2+40=-24+40=16.$$

Let us compute $p_g(S).$ We have that $p_g(S)\geq p_g\left(X_{40}\right)=5,$ thus $q(S)\geq 4.$
Suppose that $q(S)\geq 5$.
We know from \cite[Beauville Appendix]{De} that one always has $p_g(S)\geq 2q(S)-4,$ with equality
only if $S$ is the product of a curve of genus $2$ and a curve of genus $q(S)-2\geq 2.$
Thus $p_g(S)=q(S)+1$ implies that $q(S)=5,p_g(S)=6$ and $S$ is the product of a genus $2$ curve with a genus $3$ curve.
The restriction of the canonical map of $S$ to a genus $2$ fibre $F$ is a map of degree $2$ to $\mathbb P^1,$
the canonical map of $F$.
Hence the map $\pi|_F$ is of degree $\geq 2$ to $\mathbb P^1.$
This is a contradiction because, since $X_{40}$ is of general type, it is not a ruled surface.
\end{proof}

\begin{proposition}
The surface $S$ is not covered by the bidisk $\mathbb H\times\mathbb H.$
\end{proposition}

\begin{proof}
If $S$ is universally covered by $\mathbb H\times\mathbb H,$ then it is the quotient of $\mathbb H\times\mathbb H$ by a discrete
cocompact subgroup $\Gamma$ of
$\text{Aut}(\mathbb H\times\mathbb H)=\text{Aut}(\mathbb H)^2\rtimes (\mathbb Z/2\mathbb Z)$
acting freely. 
Let $$\Gamma_0:=\Gamma\cap\text{Aut}(\mathbb H)^2$$
and $\Gamma_0', \Gamma_0''$ be the projections of $\Gamma_0$ to the factors of $\text{Aut}(\mathbb H)\times\text{Aut}(\mathbb H).$
By \cite[Theorem 1]{Sh}, if one of $\Gamma_0',$ $\Gamma_0''$ is discrete, so is the other. In this case we say that
$\Gamma$ is {\it reducible}.

If $\Gamma$ is irreducible, we know from \cite[page 419]{MS} that
$$b_1(\mathbb H\times\mathbb H/\Gamma_0)=b_1(\mathbb P^1\times\mathbb P^1)=0.$$
This is impossible because $2q=b_1$ and $q(S)=4.$

So, $\Gamma$ is reducible and then $\Gamma_0$ is a finite index subgroup of $\Gamma_0'\times\Gamma_0''.$ It follows that
$\mathbb H\times\mathbb H/\Gamma_0$ is a covering of the product of two curves $\mathbb H/\Gamma_0'\times\mathbb H/\Gamma_0''.$
We {\em claim} that $S$ is {\it isogenous} to a product of curves
(i.e. it is a quotient of a product of curves by a fixed-point free group action).
In fact, there exists a normal sub-lattice $\Gamma_1$ of $\Gamma_0,$ of finite index, of the form
$$\Gamma_1'\times\Gamma_1''\subset\Gamma_0\subset\Gamma_0'\times\Gamma_0''.$$
This implies the existence of an \'etale map
$$\mathbb H\times\mathbb H/\Gamma_1=
\mathbb H/\Gamma_1'\times\mathbb H/\Gamma_1''
\longrightarrow\mathbb H\times\mathbb H/\Gamma_0,$$
the action being given by $\Gamma_0/\Gamma_1.$

Surfaces with $p_g=5$ and $q=4$ isogenous to a product of curves are classified in \cite{BNP}.
They are of the form $(C\times H)/(\mathbb Z/2\mathbb Z),$ where:
\begin{description}
\item[a)] $C$ and $H$ are curves of genus $3$ with fixed-point free involutions, or
\item[b)] $C$ is a curve of genus $5$ with a fixed-point free involution and $H$ is a bielliptic curve of genus $2$.
\end{description}
We know from \cite[Theorem 3.4]{Po} that the curves in a) are hyperelliptic, hence in both cases the canonical map
factors through a double covering of a ruled surface and then the canonical image is not of general type.
This implies that $S$ is not isogenous to a product of curves, a contradiction.
\end{proof}

\subsection{The degeneration}\label{degn}

\begin{proposition}
The family of surfaces $S$ constructed above coincides with the family of surfaces constructed by Schoen in \cite{Sc}.
\end{proposition}

\begin{proof}
The deformation space in Schoen's construction is locally irreducible (see Theorem \ref{thmSc}), hence we get from
Proposition \ref{propCMR} that it suffices to show that the $4$-dimensional family of surfaces $X_{40}$ degenerates
to a $3$-dimensional family of reducible surfaces which are the union of a double quadric surface and a quartic Kummer surface,
glued along a trope of the Kummer surface.

Let us consider the Igusa quartic given by the equation
\begin{equation*}
\begin{split}
4(x^4+y^4+z^4+w^4+t^4+h^4) & = (x^2+y^2+z^2+w^2+t^2+h^2)^2, \\
h: & = -x-y-z-w-t,
\end{split}
\end{equation*}
in $\mathbb P^4(x,y,z,w,t).$
Recall from section \ref{40nodal} that to a generic hyperplane section $$H_{a,b,c,d}:=ax+by+cz+dw-t$$
of the Igusa quartic corresponds a quadric $Q_{a,b,c,d}$ such that $X_{40}:=I_4\cap Q_{a,b,c,d}$
(notice that the two $I_4$ that appear in Figure \ref{fig1} denote isomorphic surfaces that are given by different equations in our computations).
Let $$F_{a,b,c,d}=c_1x^2+\cdots+c_{15}wt,\ \ \ c_i=c_i(a,b,c,d),$$
be the defining polynomial of $Q_{a,b,c,d}$ in $\mathbb P^4.$
The correspondence $$H_{a,b,c,d}\mapsto Q_{a,b,c,d}$$ is a rational map
$$\varphi:\mathbb A^4\dashrightarrow\mathbb P^{14}=\mathbb{P}(H^0(\mathbb{P}^4,\mathcal{O}(2))),\ \ \ (a,b,c,d)\longmapsto (c_1:\cdots:c_{15}).$$
%
%$$\begin{array}{rccc}
%  \varphi:& \mathbb A^4 & \xdashrightarrow{\text{\ \ \ \ \ \ \ \ \ }} & \mathbb P^{14} \\
%  &(a,b,c,d) & \xmapsto{\ \ \ \ \ \ \ } & (c_1:\cdots:c_{15}).
%\end{array}$$
%
From our computations (see \cite{Ri}), if $H_{a,b,c,d}$ gives a trope of $I_4,$
then $F_{a,b,c,d}$ vanishes identically.
This happens for instance for $$(a,b,c,d)=(0,0,-1,-1).$$
We resolve the corresponding indeterminacy of $\varphi$ by blowing up:
locally, this is done by evaluating the coefficients $c_i$ at \mbox{$(a,ab,ac-1,ad-1).$}
The computations give that $$F_{a,ab,ac-1,ad-1}=a^3\cdot G_{a,b,c,d},$$ with $G$ of degree $2.$
Moreover, there exists a linear form $J_{b,c,d}$ such that $$G_{0,b,c,d}=x\cdot J_{b,c,d}.$$
The hyperplane $\{x=0\}$ gives a trope of $I_4$ (a double quadric). For generic values of the parameters,
the hyperplane given by $J_{b,c,d}$ is tangent to $I_4$ at a point (it gives a quartic Kummer surface)
and the quadric given by $G_{a,b,c,d}$ meets $I_4$ at a $40$-nodal surface.
One can verify from the equations obtained in \cite{Ri} that these quadric and Kummer are
glued along a trope of the Kummer surface.
\end{proof}

\section{The surface $\overline{X}_{40}$ with $\Sigma_{5}$ symmetries}

In this section we study a surface $\overline{S}$ which is the double
cover of a particular $40$-nodal degree $8$ complete intersection
surface with a high group of symmetries.
Using these symmetries we prove that its Picard number is maximal and we find the isogeny
class of its Albanese variety. We moreover describe another construction
of a $3$-dimensional subfamily of Schoen surfaces as bidouble covers
of some special Kummer surfaces.

\subsection{Some Schoen surfaces as bidouble covers}

Recall from Section \ref{2div} the complete intersection $\overline{X}_{40}\subset\mathbb P^4$
of the following quadric and quartic: 
$$5\left(x^{2}+\cdots+t^{2}\right)-7(x+\cdots+t)^{2}=0,$$
$$4\left(x^{4}+\cdots+t^{4}+h^{4}\right)-\left(x^{2}+\cdots+t^{2}+h^{2}\right)^{2}=0,$$
where $h=-(x+y+z+w+t)$. The surface $\overline{X}_{40}$ has 40 nodes
(defined over the field $\mathbb Q(\sqrt{-15})$). The permutation group
$\Sigma_{5}$ is a subgroup of ${\rm aut}(\overline{X}_{40})$, the automorphism
group of $\overline{X}_{40}$. 

Let $\overline{S}\to\overline{X}_{40}$ be the double cover branched
over the $40$ nodes and let $\sigma$ be the corresponding involution of $\overline{S}.$ 
Let $\hat{X}_{40}$ be the minimal resolution  of $\overline{X}_{40}$.
By the argument given in the proof of Proposition \ref{2div1}, the integral cohomology group $H^2(\hat{X}_{40},\mathbb{Z})$  is torsion free. Thus the N\'eron-Severi group  $NS(\hat{X}_{40})$ being a  subgroup of $H^2(\hat{X}_{40},\mathbb{Z})$ is also torsion free.
We note also that any automorphism in $\Sigma_{5}$ preserves the
set of nodes. Then by \cite[\S 1.3, Theorem 1 e)]{Li}, each element
of $\Sigma_{5}$ lifts to an automorphism of $\overline{S}$. We have:
\begin{proposition}
\label{prop:action transposition}
Let $\tau\in\Sigma_{5}$ be a transposition. The quotient surface
$$Q:=\overline{X}_{40}/\tau$$ is a $K3$ surface with $15$
nodes containing in the smooth locus two $(-2)$-curves $A_{16}$ and $A_{16}'$
such that $A_{16}A_{16}'=10$. The double cover $\overline{X}_{40}\to Q$ is
branched over $A_{16}+A_{16}'$. \\
Let $A_{1},\dots,A_{15}$ be the $15$ \textup{$(-2)$}-curves in
the resolution $\hat{Q}$ of $Q$. The divisors $A_{16}+\sum_{i=1}^{15}A_{i}$
and $A_{16}'+\sum_{i=1}^{15}A_{i}$ are $2$-divisible. The bidouble
cover $\hat{S}\to\hat{Q}$ associated to the divisors
\[
D_{1}=\sum_{i=1}^{15}A_{i},\,D_{2}=A_{16},\,D_{3}=A_{16}'
\]
gives the blow-up $\hat{S}\rightarrow\overline{S}$ at the $40$ fixed
points of $\sigma$; the bidouble cover decomposes as
\[
\arraycolsep=6pt\def\arraystretch{1.2}
\begin{array}{ccccc}
 &  & \hat{S}\\
 & \swarrow & \downarrow & \searrow\\
\hat{B}_{1} &  & \hat{X}_{40} &  & \hat{B}_{2}\\
 & \searrow & \downarrow & \swarrow\\
 &  & \hat{Q}
\end{array}
\]
where $\hat{B}_{1},\hat{B}_{2}$ are Abelian surfaces $B_1$, $B_2$ blown-up at their $2$-torsion points, each map $\hat{S}\to\hat{B_{i}}$
is a double cover branched over a curve of genus $4$, and the maps
$\hat{B}_{i}\to\hat{Q}$, $i=1,2$ are branched over $D_{1}+D_{2}$
and $D_{1}+D_{3},$ respectively. \\
The group generated by the lifts of $\tau$ on $\overline{S}$ is
$(\mathbb Z/2\mathbb Z)^{2}$ and it contains $\sigma$.
\end{proposition}
\begin{proof}
Let $\tau\in\Sigma_{5}$ be a transposition (for example the one exchanging
the coordinates $x$ and $y$). Using Magma, we compute that the
fixed point set of $\tau$ is a union of two smooth genus 0 curves
meeting at $10$ points which are $10$ nodes of $\overline{X}_{40}$.
Moreover the quotient of the surface $\overline{X}_{40}$ by $\tau$
 is a quartic $K3$ surface $Q\hookrightarrow\mathbb P^{3}$ which has 
$15$ nodes (see \cite[Section D]{Ri}). \\ % Appendix \ref{KumQuo}). \\
The image of the fixed point set of $\tau$ by the quotient map is
$A_{16}+A_{16}',$ where $A_{16}$ and $A_{16}'$ are two $(-2)$-curves
which are disjoint from the $15$ nodes and that meet transversally at $10$ points.
It is the intersection of $Q$ with a quadric in $\mathbb P^3.$
\\
Let $A_{1},\dots,A_{15}$ be the $15$ $(-2)$-curves above the nodes
on the minimal resolution $\hat{Q}$ of $Q$. Let us keep the same
notations for the strict transform of $A_{16},A_{16}'$ on $\hat{Q}$.
The $16$ curves $A_{1},\dots,A_{15},A_{16}$ are disjoint and so
are the $16$ curves $A_{1},\dots,A_{15},A_{16}'$. Thus by \cite[Theorem 1]{Ni1},
the divisors $A_{16}+\sum_{i=1}^{15}A_{i}$ and $A_{16}'+\sum_{i=1}^{15}A_{i}$
are $2$-divisible. Using the three divisors $D_{1},D_{2},D_{3}$,
the associated bidouble cover $\hat{S}\to\hat{Q}$ gives the blow-up
of $\overline S$ at the $40$ fixed points of $\sigma$ (see \cite{Pardini_abelian} or \cite{Catanese_bidouble} for information on bidouble covers);
the remaining assertions follow.
\end{proof}

\begin{remark}\label{RemSubfamily}
More generally, one can prove that there exists a $3$-dimensional
family of quartic $K3$ surfaces with $15$ nodes, containing on their
smooth locus two $(-2)$-curves $A_{16},A_{16}'$ such that $A_{16}A_{16}'=10$ (cf. \cite{Re, Pi}).
Their associated bidouble covers as above give a $3$-dimensional subfamily of Schoen surfaces.\\
 It is interesting to compare this construction of Schoen surfaces
by bidouble covers with the construction of Lagrangian surfaces done
by Bogomolov and Tschinkel in \cite[Sections $3\ \&\ 4$]{BT}.
\end{remark}

\subsection{The $240$ automorphisms of $\overline{S}$}

We will use standard results in representation theory for which we refer the reader to \cite{Fulton_Harris}. 
The permutation group $\Sigma_{5}$ has $7$ irreducible representations
(up to isomorphism), which we denote by 
\[
U,\ U',\ V,\ V'=V\otimes U',\ W,\ W'=W\otimes U',\ \wedge^{2}V,
\]
of respective dimension $1,1,4,4,5,5,6$, where $U$ is the trivial representation, $U'$ is the signature,
the $4$-dimensional representation $V$ satisfies ${\rm Tr}(\tau)=2$ and
the $5$-dimensional representation $W$ is determined by ${\rm Tr}(\tau)=1$
({\rm Tr} is the trace and $\tau\in\Sigma_{5}$ is a transposition). 

One has $K_{\overline{X}_{40}}=\mathcal{O}(1)$. By looking at
the symmetries of the equations of $\overline{X}_{40},$ the representation of $\Sigma_{5}$
on $H^{0}\left(\overline{X}_{40},K_{\overline{X}_{40}}\right)$ is faithful. On $\mathbb{P}^4$, the point $(1:1:1:1:1)$ is invariant, thus the corresponding vector space is stable and the representation is thus not irreducible. The only non-irreducible 5-dimensional faithful representations are:

\[
U+V,\ U+V',\ U'+V,\ U'+V'.
\]
Let ${\rm aut}(\overline{S})^{o}$
be the subgroup of ${\rm aut}(\overline{S})$ generated by the lifts of
the elements of $\Sigma_{5}\subset{\rm aut}(\overline{X}_{40})$. There
is a natural exact sequence 
\[
0\to\mathbb Z/2\mathbb Z\to{\rm aut}(\overline{S})^{o}\to\Sigma_{5}\to0
\]
where the morphism $\mathbb Z/2\mathbb Z\to{\rm aut}(\overline{S})^{o}$ is obtained
by the inclusion of $\sigma$. By Schur theory, the group extensions
\[
0\to\mathbb Z/2\mathbb Z\to H\to\Sigma_{5}\to0
\]
of $\Sigma_{5}$ by $\mathbb Z/2\mathbb Z$ are classified by the second homology
group 
\[
H^{2}(\Sigma_{5},\mathbb Z/2\mathbb Z),
\]
which is isomorphic to $(\mathbb Z/2\mathbb Z)^{2}$, therefore ${\rm aut}(\overline{S})^{o}$
is one of the following groups 
\[
\mathbb Z/2\mathbb Z\times\Sigma_{5},\ 2.\Sigma_{5}^{-},\ 2.\Sigma_{5}^{+},\ 4.A_{5},
\]
which we will describe later.
\begin{theorem}
\label{Thm:The-lifting-of}The group ${\rm aut}(\overline{S})^{o}$ is
$2.\Sigma_{5}^{+}$.
\end{theorem}
We prove this result by showing that ${\rm aut}(\overline{S})^{o}$ cannot
be $\mathbb Z/2\mathbb Z\times\Sigma_{5},\,2.\Sigma_{5}^{-}$ and $4.A_{5}$.
We need the following Lemma.
\begin{lemma}
The trace of the involution $\sigma$ on $H^{0}\left(\overline{S},\Omega_{\overline{S}}\right)$ is $-4$.
\end{lemma}
\begin{proof}
The minimal resolution $\hat X_{40}$ of the quotient surface $\overline S=\overline X_{40}/{\sigma}$ is regular. By \cite[Lemma VI.11 and Example VI.12, 3)]{BeauvilleBook}, the space of $\sigma$-invariant $1$-forms on $\overline{S}$ and the space of $1$-forms on $\hat{X}_{40}$ have the same dimension. Therefore $\sigma$ acts on $H^{0}\left(\overline{S},\Omega_{\overline{S}}\right)$ by multiplication by $-1$, thus the result.
\end{proof}
We remark moreover that the morphism 
\[
\varphi_{2,0}:\wedge^{2}H^{0}\left(\overline{S},\Omega_{\overline{S}}\right)\longrightarrow H^{0}\left(\overline{S},K_{\overline{S}}\right)\simeq H^{0}\left(\overline{X}_{40},K_{\overline{X}_{40}}\right)
\]
is equivariant under ${\rm aut}(\overline S)^{o}$ and we know that it has a $1$-dimensional kernel (since it is a Schoen surface).
%(alternatively since a higher dimensional kernel would  automatically contain a non-trivial decomposable $2$-form restricting to $0$ on $\overline S$, therefore by Castelnuovo-De Franchis $\overline S$ would have a fibration onto a curve of genus $\geq 2$, which should be avoided from the classification of surfaces). It would maybe be interesting to write a full proof of that in proposition 9...
The group ${\rm aut}(\overline{S})^{o}$ acts on $H^{0}\left(\overline{S},K_{\overline{S}}\right)=H^{0}\left(\overline{X}_{40},K_{\overline{X}_{40}}\right)$
through ${\rm aut}(\overline{S})^{o}/\sigma=\Sigma_{5}$.\\

\noindent{\bf Proof of Theorem \ref{Thm:The-lifting-of}.}\newline
Suppose that ${\rm aut}(\overline{S})^{o}=\mathbb Z/2\mathbb Z\times\Sigma_{5}$.
If the $4$-dimensional representation $H^{0}\left(\overline{S},\Omega_{\overline{S}}\right)$
of $\Sigma_{5}$ is faithful, then it is $V$ or $V'$ and 
\[
\wedge^{2}H^{0}(\overline S,\Omega_{\overline S})=\wedge^{2}V=\wedge^{2}V'
\]
is an irreducible ($6$-dimensional) representation, a contradiction.
Therefore $$H^{0}\left(\overline{S},\Omega_{\overline{S}}\right)=U^{a}+U'^{b},$$
but then the representation of $\Sigma_{5}={\rm aut}(\overline{S})^{o}/\sigma$
on $\wedge^{2}H^{0}\left(\overline{S},\Omega_{\overline{S}}\right)$ is not faithful,
again a contradiction.

The group $2.\Sigma_{5}^{-}$ is the group number $89$ among the
order $240$ groups in Magma database. It contains an unique involution.
But by Proposition \ref{prop:action transposition}, the automorphisms
of $\overline{S}$ lifting the transpositions of $\Sigma_{5}$ acting
on $\overline{X}_{40}$ are involutions, thus ${\rm aut}(\overline{S})^{o}$
cannot be $2.\Sigma_{5}^{-}$.

The group $A_{5}.4$ (group number $91$ in Magma database) has $14$
irreducible representations $\chi_{i},\,i=1,\dots,14$ of respective
dimensions
\[
1^{4}, 4^{4}, 5^{4}, 6^{2},
\]
(where $a^{b}$ means $a$ repeated $b$ times).

Let $W_4$ be a non irreducible $4$-dimensional representation of $A_5 .4$. It is easy to see that $\wedge^2 W_4$ is not a faithful representation of $\Sigma_5 ={\rm aut}(\overline{S})^{o}/\sigma$, thus $\wedge^{2}H^{0}(\overline S,\Omega_{\overline S})$ cannot be such representation $W_4$.

Looking at the character table (for instance given by Magma),
the representation $H^{0}\left(\overline{S},\Omega_{\overline{S}}\right)$ cannot
be $\chi_{5}$ or $\chi_{6}$ since the trace of the involution $\sigma$
must be $-4$. The two remaining $4$-dimensional representations
$\chi_{7},\chi_{8}$ satisfy 
\[
\wedge^{2}\chi_{7}=\wedge^{2}\chi_{8}=\chi_{14},
\]
(for the computation of the wedge product of a representation see \cite{Fulton_Harris}) which is an  irreducible representation of $\Sigma_{5}={\rm aut}(\overline{S})^{o}/\sigma$,
thus $A_{5}.4$ is not ${\rm aut}(\overline{S})^{o}$. The only possibility
is thus ${\rm aut}(\overline{S})^{o}=2.\Sigma_{5}^{+}$.
\Qed

\subsection{The group $2.\Sigma_{5}^{+}$ and its action on $\overline{S}$ }

The group $2.\Sigma_{5}^{+}$ (number $90$ among groups of order
$240$ in Magma database) has $12$ irreducible representations $\chi_{1},\dots,\chi_{12}$,
of respective dimensions 
\[
1^{2}, 4^{5}, 5^{2}, 6^{3}.
\]
It has $21$ involutions, divided into two conjugacy classes, one
containing an unique element $\sigma$, which is the involution of
the double cover $\overline{S}\to\overline{X}_{40}$. Since the trace
of $\sigma$ on $\chi_{3}$ and $\chi_{5}$ is not $-4$, the only
possibilities are $H^{0}\left(\overline{S},\Omega_{\overline{S}}\right)=\chi_{4},\,\chi_{6}$
or $\chi_{7}$. One has
\[
\wedge^{2}\chi_{4}=\chi_{1}+\chi_{2}+\chi_{3}\mbox{ \ \ and \ \ }\wedge^{2}\chi_{6}=\wedge^{2}\chi_{7}=\chi_{2}+\chi_{9}.
\]
The representation of $2.\Sigma_{5}^{+}$ on $\chi_{9}$ gives an
irreducible $5$-dimensional representation of $2.\Sigma_{5}^{+}/\sigma=\Sigma_{5}$,
which is impossible since $H^{0}\left(\overline{X}_{40},K_{\overline{X}_{40}}\right)$
is not irreducible. We thus proved that $H^{0}(\overline{S},\Omega_{\overline{S}})=\chi_{4}$,
which has character
\[
\begin{array}{ccccccccccccc}
Order & 1 & 2 & 2 & 3 & 4 & 5 & 6 & 6 & 6 & 8 & 8 & 10\\
Trace & 4 & -4 & 0 & -2 & 0 & -1 & 0 & 0 & 2 & 0 & 0 & 1
\end{array}
\]
We conclude that:
\begin{proposition}
The representation of the group $2.\Sigma_{5}^{+}$ on $H^{0}\left(\overline{S},\Omega_{\overline{S}}\right)$
is $\chi_{4}$. Moreover, one has $\wedge^{2}H^{0}\left(\overline{S},\Omega_{\overline{S}}\right)=\chi_{1} + \chi_{2} + \chi_{3}$
and 
\[
H^{1,1}(A)=\chi_{4}\otimes\chi_{4}=\chi_{1}+\chi_{2}+\chi_{3}+\chi_{5}+\chi_{10},
\]
where $A$ is the Albanese variety of $\overline{S}$.
\end{proposition}

The group $\Sigma_{5}={\rm aut}(\overline{S})^{o}/\sigma$ acts on $\wedge^{2}\chi_{4}$
and $\wedge^{2}\chi_{4}=U + U'+ V$.
\begin{proposition}
The representation of $2.\Sigma_{5}^{+}$ on $H^{0}\left(\overline{S},K_{\overline{S}}\right)$
is $\chi_{2}+\chi_{3}$.\end{proposition}
\begin{proof}
The trace of an involution $\iota\neq\sigma$ in $2.\Sigma_{5}^{+}$
acting on $\chi_{4}$ equals $0$, thus the eigenvalues of $\iota$ on the space of holomorphic one forms are $1,1,-1,-1$. 
Moreover, since $\wedge^{2}\chi_{4}=\chi_{1}+\chi_{2}+\chi_{3}$,
the involution $\iota$ acts on $H^{0}\left(\overline{S},K_{\overline{S}}\right)$
with trace $-1$ or $-3$ according if 
\[
H^{0}\left(\overline{S},K_{\overline{S}}\right)=U+V\mbox{ \ \ or \ \ }H^{0}\left(\overline{S},K_{\overline{S}}\right)=U'+V.
\]
Then the eigenvalues of $\iota$ on $H^{0}\left(\overline{S},K_{\overline{S}}\right)$ are respectively $1,1,-1,-1,-1$ and $1,-1,-1,-1,-1$.
By \cite[Lemma VI.11 and Example VI.12, 3)]{BeauvilleBook}, the quotient surface has invariants $q=2$ and $p_{g}=2$ or $p_{g}=1$ respectively.
%In the first case $\dim H^{0}\left(\overline{S},K_{\overline{S}}\right)^{\iota}=2$.
%In the second case, one has  $\dim H^{0}\left(\overline{S},K_{\overline{S}}\right)^{\iota}=1$
%and by \cite[Lemma VI.11 and Example VI.12, 3)]{BeauvilleBook}, the quotient surface has $q=2,\,p_{g}=1,\chi=0$. 
By Proposition
\ref{prop:action transposition}, that quotient surface is (birational to) an Abelian surface and it is the second case that is
actually occurring, thus $H^{0}\left(\overline{S},K_{\overline{S}}\right)=U'+V,$
which corresponds to the representation $\chi_{2}+\chi_{3}$ for $2.\Sigma_{5}^{+}$.
\end{proof}

There is a basis  $\omega_{1},\dots,\omega_{4}$ of $H^{0}\left(\overline{S},\Omega_{\overline{S}}\right)$
such that the action of $2.\Sigma_{5}^{+}$ is generated by the following
matrices of order $2$ and $8$:
\begin{equation}
\left(\begin{array}{cccc}
0 & 1 & 0 & 0\\
1 & 0 & 0 & 0\\
0 & 0 & 0 & 1\\
0 & 0 & 1 & 0
\end{array}\right),\,\left(\begin{array}{cccc}
\frac{\sqrt{2}}{2}(1+I) & 0 & -\frac{\sqrt{2}}{2}(1+I) & -I\\
0 & 0 & -1 & 0\\
0 & 1 & -\sqrt{2} & -1\\
0 & 0 & 0 & \frac{\sqrt{2}}{2}(1-I)
\end{array}\right),\label{eq:matrices}
\end{equation}
where $I^2=-1$.\newline
We have 
\[
\wedge^{2}H^{0}(\overline S, \Omega_{\overline S})=\chi_{1}+\chi_{2}+\chi_{3},
\]
where  the trivial representation $\chi_{1}$ is generated by the
indecomposable vector
\[
v=\omega_{1}\wedge\omega_{4}+\omega_{2}\wedge\omega_{3}
\]
which generates the kernel of $\wedge^{2}H^{0}\left(\overline S, \Omega_{\overline{S}})\to H^{0}(\overline{S},K_{\overline{S}}\right)$.
By the theorem of Castelnuovo-de Franchis, that gives another proof
that $\overline{S}$ has no fibration onto a curve of genus $\geq2$.

\subsection{The periods of the Albanese variety of $\overline S$}\label{periods}

Let us study the Albanese variety of $\overline{S}.$
\begin{proposition}
The Albanese variety $A$ of $\overline{S}$ is isogenous to $E^{4}$
where $E$ is an elliptic curve with CM by $\mathbb Z[\sqrt{-15}]$.\end{proposition}
\begin{proof}
The Albanese variety $A$ of $\overline{S}$ is $A=H^{0}(\overline S,\Omega_{\overline S})^{*}/\Lambda$,
where $\Lambda=H_{1}\left(\overline{S},\mathbb Z\right)\subset H^{0}\left(\overline{S},\Omega_{\overline{S}}\right)^{*}$.
Since $2.\Sigma_{5}^{+}$ acts on $A$, $\Lambda$ is a $2.\Sigma_{5}^{+}$-stable
lattice in $\chi_{4}=H^{0}\left(\overline{S},\Omega_{\overline{S}}\right)^{*}$.
The representation $\chi_{4}$ has Schur index $2$ and one computes
that there exists a non-trivial $2.\Sigma_{5}^{+}$-invariant anti-symmetric
bilinear form on $V_{4}=\chi_{4}$. By \cite[Theorem 4.1 $(ii_4)$]{PZ},
that implies that $A$ is isogenous to $E^{4}$ where $E$ is an elliptic
curve with CM. 

Let $\tau$ be the involution acting on $\overline{X}_{40}$ by exchanging the first two coordinates. The line 
\[
L=\left\{X+Z=Y+\frac{1}{4}(-1+I\sqrt{15})W=0\right\},\,\,I^{2}=-1
\]
 is contained in the quotient surface $Q=\overline{X}_{40}/\tau$, the equation of which is given in \cite{Ri}.
This line contains $3$ nodes $a_{1},a_{2},a_{3}$, and cuts the two
$(-2)$-curves disjoint from the $15$ nodes in points denoted by
$a_{4}$ and $a_{5}$. By Proposition \ref{prop:action transposition}
and its proof, the surface $Q$ is birational to two Kummer surfaces
$B_{i}/[-1]$, $i=1,2$, where each $B_{i}$ is an Abelian surface.
The $4$ points $a_{1},\dots,a_{4}$ are the branch points of a degree
$2$ cover $E\to L$, where $E$ is therefore an elliptic curve on
$B_{1}$ (say). The line $L$ is also the image of an elliptic curve
on $B_{2}$, the branch points being $a_{1},a_{2},a_{3},a_{5}$. Using
cross ratio for the points $a_{1},\dots,a_{4}$, one finds that 
\[
E=\{y^{2}=x(x-1)(x-\lambda)\}
\]
 where 
\[
\lambda=\frac{1}{64}\left(17+21\sqrt{5}+I\left(7\sqrt{15}-17\sqrt{3}\right)\right),\,\,I^{2}=-1.
\]
The $j$-invariant of $E$ is 
\[
-\frac{3^{3}5}{2}\left(5\cdot283 + 7^{2}13\sqrt{5}\right)
\]
 and using Magma again (see \cite[Section E]{Ri})
 , % Appendix \ref{elliptic_curve}, 
 one obtains that $E$ has CM by the order
$\mathbb Z[\sqrt{-15}]$ (taking the cross ratio for $a_{1},a_{2},a_{3},a_{5}$
gives an elliptic curve whose $j$-invariant is conjugated to $j(E)$,
having CM by the same order). We know by Proposition \ref{prop:action transposition}
that $\overline{S}$ admits a map onto $B_{1}$, thus the result.

\end{proof}

\subsection{The surface $\overline{S}$ has maximal Picard number}

Finally we prove the following.
\begin{theorem}
The surface $\overline{S}$ and the minimal resolution $\hat{X}_{40}$
of $\overline{X}_{40}$ have maximal Picard number, equal respectively
to $12$ and $52$.
\end{theorem}
\begin{proof}
The Albanese variety $A$ of $\overline{S}$ is isogeneous to $E^{4}$
where $E$ is an elliptic curve with CM, therefore $A$ has maximal
Picard number. Moreover the map
\[
H^{2,0}(A)=\wedge^{2}H^{0}\left(\overline{S},\Omega_{\overline{S}}\right)\longrightarrow H^{2,0}\left(\overline{S}\right)=
H^{0}\left(\overline{S},K_{\overline{S}}\right)
\]
is surjective, thus by \cite[Proposition 2(a)]{Be2}, the
surface $\overline{S}$ has maximal Picard number. There is a dominant
rational map $\overline{S}\dashrightarrow\hat{X}_{40}$ and $\overline{S}$
has maximal Picard number, thus by \cite[Proposition 2(b)]{Be2},
the surface $\hat{X}_{40}$ has maximal Picard number. It is easy
to check that $h^{1,1}(\overline{S})=12$ and $h^{1,1}(\hat{X}_{40})=52$.
\end{proof}

\appendix
\section{Appendix: Quartics with 15 nodes}\label{Sarti}
%\subsection{Quartics with 15 nodes}\label{Sarti}
%
Let $Q_{15}$ be a $K3$ surface in $\mathbb P^3(\mathbb C)$ with 15 nodes.
In this section we show that:
\begin{itemize}
\item The moduli space of quartic $K3$ surfaces with $15$ nodes can be described as the moduli
space of $K3$ surfaces polarized by some lattice $N$ that we describe below, and it is irreducible; 
\item A generic $K3$ surface with $15$ nodes can be realized as a section of the Igusa quartic threefold,
generalizing a similar result for Kummer quartic surfaces.
\end{itemize}
%%%%%%%%%%%%%%%%%%%%%%%%%%%%%%%%%%%%%%%%%%%%%
%
%
The K3 surfaces as $Q_{15}$ above are described in \cite[Theorem 8.6]{GS} and in \cite[Section 5]{Ga}, we recall here the 
following result for convenience:
\begin{theorem}[\cite{GS}]\label{teoconali}
Let $\tilde{Q}_{15}$ be a projective $K3$ surface with 15 disjoint smooth rational curves $M_i$, $i=1,\ldots, 15$. Then:
\begin{itemize}
\item[1)]The N\'eron-Severi group of $\tilde{Q}_{15}$ contains the lattice $M_{(\mathbb Z/2\mathbb Z)^4}$ (which
is the smallest primitive sublattice of the $K3$ lattice containing the 15 rational curves $M_i$);
\item[2)] there exists a $K3$ surface $X$ with a symplectic action by $G=(\mathbb Z/2\mathbb Z)^4$ such that
$\tilde{Q}_{15}$ is the minimal resolution of the quotient $X/G$.
\end{itemize} 
\end{theorem}
With the same notations as in Theorem \ref{teoconali}, assume that $\tilde{Q}_{15}$ is the minimal resolution of $Q_{15}$. By \cite[Theorem 8.3]{GS} the N\'eron-Severi group $NS(\tilde{Q}_{15})$ contains the sublattice $\langle 4\rangle \oplus \langle -2\rangle^{\oplus 15}$ of rank 16. We denote by $M_1,\ldots, M_{15}$ the fifteen $(-2)$--curves that are the exceptional divisors on $\tilde{Q}_{15}$.  In the next section we show that $NS(\tilde{Q}_{15})$ must contain a special overlattice of $\langle 4\rangle \oplus \langle -2\rangle^{\oplus 15}$, which  is described in details in \cite[Theorem 8.3]{GS} and is generated by:
\begin{itemize}
\item A pseudo-ample class $L$ with $L^2=4$ (and $L\cdot M_i=0$, $i=1,\ldots, 15$);
\item the lattice $M:=M_{(\mathbb Z/2\mathbb Z)^4}$ (that we recall below);
\item a class $(L-v)/2$ where $v$ contains exactly $6$ of the $M_i$'s in its support (these are not arbitrarily
chosen and we recall them below). 
\end{itemize}
%%%%%%%%%%%%%%%%%%%%%%%%%%%%%%%%%%%%%%%%%%%%%%%%%%%%%%
\subsection{The lattice $M$ and the class $v$}\label{classes}
The lattice $M$ has discriminant $2^7$ and it is decribed by Nikulin \cite[\S 7]{Ni2}. Let $K$ denote the Kummer lattice,
i.e. the smallest sublattice of the $K3$ lattice that contains sixteen $(-2)$--classes. This is negative definite, has rank 16
and discriminant $2^6$, see \cite{Ni1}. We identify the 16 classes of the Kummer lattice with the elements of $(\mathbb Z/2\mathbb Z)^4$ so 
we denote the curves by $K_{ijkh}$ with $i,j,k,h\in\{0,1\}$. One can identify $M=K_{0000}^{\perp}\cap K$.
By using the description of $K$ (see e.g. \cite{GS}), the following classes are contained in $M$:
\begin{itemize}
\item The $15$ classes $K_{ijkh}$ with $(i,j,k,h)\in(\mathbb Z/2\mathbb Z)^4\backslash\{(0,0,0,0)\}$;
\item let $W$ be an hyperplane in the affine space $(\mathbb Z/2\mathbb Z)^4$ with an equation $\sum_{i=1}^{4}\alpha_i x_i=1$,
with $\alpha_i\in\{0,1\}$. Then the $15$ classes $(1/2)\sum_{p\in W} K_p$ are contained in $M$. Each of these
classes contains exactly 8 distinct $(-2)$--classes of the $K_{ijkh}$.
\end{itemize}
Finally as explained in \cite[Theorem 8.3]{GS} the class $v$ such that $(L-v)/2\in NS(Y)$ can be taken as the sum
$$
K_{0001}+K_{0010}+K_{0011}+K_{1000}+K_{0100}+K_{1100}.
$$
{\bf Notation :} For the rest of the section we will denote the fifteen $(-2)$--classes by $M_i$, $i=1,\ldots, 15$ or by 
$K_{ijkh}$ with $(i,j,k,h)\in(\mathbb Z/2\mathbb Z)^4\backslash\{(0,0,0,0)\},$ depending if it is important or not to specify the indices. 
%%%%%%%%%%%%%%%%%%%%
\subsection{The N\'eron-Severi group}\label{NS1}
Let $N$ denote the abstract lattice generated 
by  $\mathbb Z L\oplus M$ and by a class $(L-v)/2$. The next result is contained in the paper \cite[Proposition 5.1]{Ga} in a more general context, for convenience we give here a specific proof for our situation.  
\begin{proposition}\label{pol}
Let $\tilde{Q}_{15}$ be the minimal resolution of a $K3$ quartic surface with $15$ nodes, then $\tilde{Q}_{15}$ is pseudo-ample $N$-polarized, i.e. there is a primitive embedding of $N$ in $NS(\tilde{Q}_{15})$ and the image of $N$ in $NS(\tilde{Q}_{15})$ contains a pseudo-ample class.    
\end{proposition}
\begin{proof}
We use a similar argument as in the proof of \cite[Theorem 8.6]{GS}. By construction and by \cite[Theorem 8.6, 1)]{GS} we know that $\mathbb Z L\oplus M$ is a sublattice of $NS(\tilde{Q}_{15})$ (and $L$ is pseudo-ample). Let $Q$ be the orthogonal complement of $\mathbb Z L\oplus M$ in $NS(\tilde{Q}_{15})$ and let $R:=(\mathbb Z L\oplus M)\oplus Q$, then $NS(\tilde{Q}_{15})$ is an overlattice of finite index of $R$ and $R^{\vee}/R$ has number of generators $\ell(R)$ equal to $1+7+\ell(Q)$ where $\ell(Q)$ denotes the number of generators of $Q^{\vee}/Q$ (recall that $M$ has discriminant $2^7$ and discriminant group isomorphic to $(\mathbb Z/2\mathbb Z)^7$). If $k$ denotes the index of $R$ in $NS(\tilde{Q}_{15}),$ then we have 
$$
\ell(NS(Y))=8+\ell(Q)-2k.
$$
Let $T_{\tilde{Q}_{15}}$ be the transcendental lattice.
Since the $K3$ lattice is unimodular, we have
$$
\ell\left(NS\left(\tilde{Q}_{15}\right)\right)=
\ell\left(T_{\tilde{Q}_{15}}\right)\leq \rk\left(T_{\tilde{Q}_{15}}\right)=22-\rk\left(NS\left(\tilde{Q}_{15}\right)\right)=6-\rk(Q).
$$
This gives
$$
8+\ell(Q)-2k\leq 6-\rk(Q)
$$
and then
$$
k\geq \frac{1}{2}(\ell(Q)+\rk(Q))+1.
$$
Observe that $k$ is the minimum number of classes we have to add to $R$ to obtain the lattice $NS(\tilde{Q}_{15})$. The classes can be of two types,  either these are classes in $(\mathbb Z L\oplus M)^{\vee}/(\mathbb Z L\oplus M)$ or these are sums $\nu+\nu'$ with $\nu\in (\mathbb Z L\oplus M)^{\vee}/(\mathbb Z L\oplus M)$ and $\nu'\in Q^{\vee}/Q$. The maximum number of classes of the second kind is bounded by $\ell(Q)$ so we must have
at least $(\rk(Q)-\ell(Q))/2+1$ classes of the first type. Since $\rk(Q)-\ell(Q)\geq 0$, we have at least one class of the first kind, i.e. contained in $(\mathbb Z L\oplus M)^{\vee}/(\mathbb Z L\oplus M)$. The discriminant group here is $\mathbb Z/4\mathbb Z\oplus M^{\vee}/M=\mathbb Z/4\mathbb Z\oplus (\mathbb Z/2\mathbb Z)^7$. Observe that a class $\nu$ here
is then of the form $(aL/4+ w/2)$ and we have $2(aL/4+w/2)-w\in NS(Y)$ so that $a=\pm 2$. This shows that the class can be assumed to be $(L+ w)/2$.
Moreover the square of this class must be in $2\mathbb Z,$ that gives $L^2+w^2=0 \mod 8$. If
$h$ is the number of curves contained in the support of $w,$ we get $2-h=0 \mod 4$. By the description of the discriminant group of $M$ \cite[Proposition 8.2]{GS} we get that $h=6$ or $h=10$, so that we may assume that the class is of the form $(L-v)/2$ as in the statement (since by \cite[Proposition 8.2]{GS} if we take a class with $h=10$ we get the same lattice $N$). This concludes the proof. 
\end{proof}
\begin{remark}\label{remkX}
  One can easily show that if a $K3$ surface has N\'eron-Severi group exactly isometric to $N$, then it admits a projective model as a quartic surface with $15$ nodes (i.e. $N$ contains a pseudo-ample class), so the corresponding moduli space $X_\Gamma$  is $4$-dimensional and (see \cite[section $2.3$]{Hu})
  it is an arithmetic quotient by some subgroup $\Gamma$ of the isometries of the $K3$ lattice of the domain
  $$
\mathcal D_N=\{\omega\in\mathbb P(T\otimes \mathbb C)\,|\,\omega^2=0,\ \omega\bar{\omega}=0\}, 
$$
where $T$ is the orthogonal complement of $N$ in the $K3$ lattice $U^3\oplus E_8(-1)^2$. This has rank four  and it is the transcendental lattice of the generic $K3$ surface in the family.
\end{remark}
\subsection{The Moduli Space}
Let $\mathcal M_N$ be the moduli space of $K3$ surfaces that are pseudo-ample $N$-polarized. This moduli space is described e.g. in
\cite[Section 1]{Do1}, where it is shown that it is isomorphic to the space $X_\Gamma$ from Remark \ref{remkX}.
\begin{proposition}\label{moduli}
The moduli space $\mathcal M_N$ is irreducible.
\end{proposition}
\begin{proof}
The embedding of $N$ into the $K3$ lattice is unique by \cite[Theorem 1.14.4 and Remark 1.14.5]{Ni3} (see also \cite[Theorem 8.3]{GS}). By the construction of \cite[Section 3]{Do1}, $\mathcal{D}_N$ has two connected components both isomorphic to a bounded Hermitian domain of type $IV_{19-(\rk(N)-1)}=IV_4$. Observe that by \cite[Theorem 1.13.2 and Theorem 1.14.2]{Ni3}, the orthogonal complement of $N$ in the $K3$ lattice is uniquely determined by signature and discriminant form. We compute as in \cite[Theorem 8.3]{GS} that the discriminant group of $N$ is $(\mathbb Z/4\mathbb Z)\oplus(\mathbb Z/2\mathbb Z)^5$. If we denote by $q_2$ the discriminant form of the lattice $U(2)$ (that denotes the lattice $U$ with the bilinear form multiplied by $2$), then the discriminant form is the same as $q_2\oplus q_2$ on $(\mathbb Z/2\mathbb Z)^4$ and take value $1/4$ and $1/2$ on the remaining part $(\mathbb Z/4\mathbb Z)\oplus(\mathbb Z/2\mathbb Z)$. Hence we can identify $N^{\perp}$ (modulo isometries) with the lattice  
$$
U(2)\oplus U(2)\oplus \langle-2\rangle\oplus\langle -4\rangle.
$$ 
By \cite[Proposition 5.6 and Lemma 5.4]{Do1} there is an involution in $\Gamma$ that exchanges the two connected components 
of $\mathcal{D}_N$, so that $X_{\Gamma}\simeq \mathcal M_N$ is irreducible.
\end{proof}
Since the hyperplane sections of the Igusa quartic give a 4-dimensional family of quartic surfaces with 15 nodes,
then Proposition \ref{moduli} implies the following.
\begin{theorem}
A generic quartic $K3$ surface with 15 nodes can be realized as a section of the Igusa quartic.
\end{theorem}
\begin{remark}
An interesting loci in the moduli space $\mathcal{M}_N$ corresponds to quartic Kummer surfaces with 16 nodes, that can be described as tangent sections of the Igusa quartic, see \cite[Chapter $3$, Section $3.3.3$]{Hu}.
\end{remark}

\bibliography{References}

\

\

\noindent Carlos Rito
\vspace{0.1cm}
\\{\it Permanent address:}
\\ Universidade de Tr\'as-os-Montes e Alto Douro, UTAD
\\ Quinta de Prados
\\ 5000-801 Vila Real, Portugal
\\ www.utad.pt, {\tt crito@utad.pt}
\\{\it Temporary address:}
\\ Departamento de Matem\' atica
\\ Faculdade de Ci\^encias da Universidade do Porto
\\ Rua do Campo Alegre 687
\\ 4169-007 Porto, Portugal
\\ www.fc.up.pt, {\tt crito@fc.up.pt}\\

\noindent Xavier Roulleau
\vspace{0.1cm}
\\Aix-Marseille Universit\'e, CNRS, Centrale Marseille,
\\I2M UMR 7373, 
\\13453 Marseille, France\\
{\tt Xavier.Roulleau@univ-amu.fr}\\

\noindent Alessandra Sarti
\vspace{0.1cm}
\\Universit\'e de Poitiers\\
Laboratoire de Math\'ematiques et Applications, UMR 7348 du CNRS\\
Boulevard Pierre et Marie Curie, T\'el\'eport 2 - BP 30179\\
86962 Futuroscope Chasseneuil, France\\
{\tt sarti@math.univ-poitiers.fr}\\

\end{document}